\documentclass[10pt]{amsart}
\usepackage[a4paper,marginratio=1:1]{geometry}
\usepackage[initials,non-sorted-cites]{amsrefs}
\usepackage{tensor,braket,bm,fouridx}

\newtheorem{thm}{Theorem}[section]
\newtheorem{prop}[thm]{Proposition}
\newtheorem{lem}[thm]{Lemma}
\newtheorem{cor}[thm]{Corollary}
\theoremstyle{definition}

\theoremstyle{remark}

\newcommand{\abs}[1]{\lvert#1\rvert}
\newcommand{\Abs}[1]{\left\lvert#1\right\rvert}
\newcommand{\norm}[1]{\|#1\|}

\newcommand{\bdry}{\partial}
\DeclareMathOperator{\Vol}{Vol}
\DeclareMathOperator{\diam}{diam}
\DeclareMathOperator{\grad}{grad}

\DeclareMathOperator{\tr}{tr}
\DeclareMathOperator{\Ric}{Ric}

\numberwithin{equation}{section}

\title{Proper harmonic maps between asymptotically hyperbolic manifolds}
\author{Kazuo Akutagawa}
\address{Department of Mathematics, Tokyo Institute of Technology}
\email{akutagawa@math.titech.ac.jp}
\thanks{The first author is supported in part by the Grant-in-Aid for Challenging Exploratory Research,
	Japan Society for the Promotion of Science, No.~24654009.}
\author{Yoshihiko Matsumoto}
\address{Department of Mathematics, Tokyo Institute of Technology}
\email{matsumoto@math.titech.ac.jp}

\thanks{The second author is supported in part by the Grant-in-Aid for JSPS Fellows,
	Japan Society for the Promotion of Science, No.~26-11754.}

\begin{document}

\begin{abstract}
	Generalizing the result of Li and Tam for the hyperbolic spaces, 
	we prove an existence theorem on the Dirichlet problem for harmonic maps 
	with $C^1$ boundary conditions at infinity between asymptotically hyperbolic manifolds.
\end{abstract}

\maketitle

\section*{Introduction}

Various aspects of proper harmonic maps between noncompact manifolds are still yet to be clarified.
For the hyperbolic spaces, the asymptotic Dirichlet problem for such maps between them is
investigated in a series of papers of Li and Tam~\cite{LiTam1,LiTam2,LiTam3}
and by the first author~\cite{Akutagawa} (for the hyperbolic disks).
In this article, we shall extend their existence and uniqueness result for
general asymptotically hyperbolic manifolds.

We first set up some definitions.
Let $M$ be a noncompact manifold of dimension at least two equipped with a smooth Riemannian metric $g$,
and we assume that $M$ compactifies into a smooth manifold-with-boundary $\overline{M}=M \sqcup \bdry M$,
which is fixed implicitly.
Then $(M,g)$ is called a \emph{$C^2$ conformally compact manifold} if $r^2g$ extends to a
$C^2$ Riemannian metric $\overline{g}$ on $\overline{M}$,
where $r\in C^\infty(\overline{M})$ is a smooth (positive) boundary defining function.
The conformal class on $\bdry M$ represented by $\overline{g}|_{T\bdry M}$ is called the conformal infinity of $g$,
which we assume is smooth for simplicity. If in addition $(M,g)$ satisfies
\begin{equation}
	\label{eq:asymptotically_hyperbolic}
	\abs{d\log r}_{g}=\abs{dr}_{\overline{g}}=1\qquad\text{on $\bdry M$},
\end{equation}
which is equivalent to that the sectional curvature $K_g$ uniformly tends to $-1$ at the boundary
(see~\cite{Mazzeo}),
then $(M,g)$ is said to be \emph{asymptotically hyperbolic}, or just AH for brevity.

Suppose we are given two $C^2$ conformally compact AH manifolds $(M,g)$ and $(N,h)$,
whose dimensions are $m+1$ and $n+1$, where $m$, $n\ge 1$.
For any boundary map $f\in C^0(\bdry M,\bdry N)$, we consider the space of its extensions to the whole manifold
$\overline{M}$:
\begin{equation*}
	\mathcal{M}_f
	=\set{u\in C^0(\overline{M},\overline{N})|\text{$u$ maps $\bdry M$ into $\bdry N$ and $u|_{\bdry M}=f$}}.
\end{equation*}
If $\mathcal{M}_f$ is nonempty, then each of its connected components is called a
\emph{relative homotopy class}.
We look for a harmonic map $u\in C^\infty(M,N)$,
which has by definition vanishing tension field $\tau(u)$, in a given such class.
Now our main theorem, which extends a result of Li and Tam~\cite[Theorem 6.4]{LiTam3}, is stated as follows.
This can also be regarded as a noncompact version of the celebrated result of Eells--Sampson~\cite{EellsSampson}.

\begin{thm}
	\label{thm:existence}
	Let $(M,g)$, $(N,h)$ be as above, and assume that $g$ satisfies
	\begin{equation}
		\label{eq:asymptotic_Einstein}
		\abs{\Ric(g)+ng}_g=o(r)\qquad\text{as $r\to 0$}
	\end{equation}
	and $h$ has nonpositive sectional curvature.
	If $u_0\in C^1(\overline{M},\overline{N})$ satisfies $u_0(\bdry M)\subset\bdry N$ and
	$f=u_0|_{\bdry M}$ has nowhere vanishing differential $df\colon T\bdry M\longrightarrow T\bdry N$,
	then there exists a unique proper harmonic map $u\in C^1(\overline{M},\overline{N})\cap C^\infty(M,N)$
	within the same relative homotopy class.
\end{thm}

The curvature conditions on $g$ and $h$ are imposed merely for technical reasons and we do not know whether
they are necessary for the conclusion to hold.
While the nonpositivity of $K_h$ is a strict condition,
the asymptotic Einstein condition \eqref{eq:asymptotic_Einstein} is
not too restrictive (see Lemma \ref{lem:CDLS}).

There are several preceding works regarding harmonic maps between AH manifolds.
Let $g$ and $h$ be smooth conformally compact.
Under this assumption, Leung showed in his thesis~\cite{Leung} that,
if $(N,h)$ has negative sectional curvature and $v\in C^{2,\alpha}(\overline{M},\overline{N})$
satisfies $\abs{\tau(v)}=O(r^\nu)$ for some $\nu>0$,
then the solution of the heat equation with initial data $v$ converges to a harmonic map $u$
with $d_h(u,v)=O(r^\theta)$ for some $\theta>0$.
Economakis, again in his thesis~\cite{EconomakisThesis}, proved that any $C^{1,1}$ proper harmonic map whose
boundary map is smooth and has nowhere vanishing energy density admits ``polyhomogeneous'' expansion
at the boundary.
Recently, standing on a work of Donnelly~\cite{Donnelly},
Fotiadis~\cite{Fotiadis} took a general approach on noncompact complete manifolds
using an estimate of Green's function of $(M,g)$ to show that
the existence of a harmonic map boils down to that of an approximate solution $v$
and a positive lower bound of the spectrum of $(M,g)$.
For AH manifolds such a spectrum bound is established by Mazzeo~\cite{Mazzeo},
so we may apply Fotiadis' result to almost recover Leung's one.

Compared to the above-mentioned existing works, this article has two features.
First, we carry out the construction of a good approximate solution $v$.
This is necessary for establishing existence results, but it has been missing until now.
Second, in our assumption, the Dirichlet data $f$ has only $C^1$ regularity.
This makes the works of Leung and Fotiadis inapplicable, so
the process of turning $v$ into a genuine harmonic map $u$ is again discussed in a different way.
We are interested in this weak regularity assumption because
it is supposedly the critical regularity to assert the uniqueness.
In fact, in the case of the hyperbolic spaces and when $m=n$, there are families of $C^\mu$ harmonic maps
for some $\mu\in(0,1/2]$ whose boundary maps are the identity map on $S^m$,
as observed by Li and Tam~\cite{LiTam2} for $m=n=1$ and by Economakis~\cite{Economakis} in general dimensions.
Throughout the whole argument, we will basically proceed by modifying the approach taken
in~\cite{LiTam3} step by step. However, in some places we really need new ideas including
application of deep analytic results on AH manifolds (see the proof of Lemma \ref{lem:barrier} for instance).

The authors would like to thank Robin Graham and Man Chun Leung for their help.

\section{Approximate solution}
\label{sec:approximate}

We shall identify a neighborhood of $\bdry M$ in $\overline{M}$ with the product $\bdry M\times[0,r^*)$ for
some $r^*>0$.
Our first lemma, which is borrowed from an article of
Chru\'sciel, Delay, Lee, and Skinner~\cite[Lemma 3.1]{CDLS}, is concerning a good identification.

\begin{lem}
	\label{lem:CDLS}
	Let $(M,g)$ be a $C^2$ conformally compact AH manifold with smooth conformal infinity $\gamma$.
	Take a smooth representative $\Hat{g}$ of $\gamma$ arbitrarily.
	If $g$ satisfies \eqref{eq:asymptotic_Einstein},
	then there exists a $C^3$ map $\Psi$ from an open neighborhood of $\bdry M$ in $\overline{M}$ onto
	$\bdry M\times[0,r^*)$ for some $r^*>0$ that restricts to the identity map on $\bdry M$ and to
	a diffeomorphism between the interiors for which
	\begin{equation*}
		r^2(\Psi^{-1})^*g=dr^2+\Hat{g}+O(r^2)\qquad\text{as $r\to 0$},
	\end{equation*}
	where $r\colon\bdry M\times[0,r^*)\longrightarrow[0,r^*)$ denotes the projection onto the second factor,
	in the sense that $\abs{r^2(\Psi^{-1})^*g-(dr^2+\Hat{g})}_{dr^2+\Hat{g}}=O(r^2)$.
	Conversely, existence of such $\Psi$ implies \eqref{eq:asymptotic_Einstein}.
\end{lem}

In~\cite[Lemma 3.1]{CDLS}, the metric $g$ is assumed to be exactly Einstein, but the proof shows that
\eqref{eq:asymptotic_Einstein} suffices. The converse is also clear from the proof.

We may even assume that $\Psi$ gives a diffeomorphism up to the boundary because we can replace the
$C^\infty$ structure of $\overline{M}$ with the one that $\Psi$ induces
and it makes no difference to the conclusion of Theorem \ref{thm:existence}.
We furthermore omit $\Psi$ and just write, for example,
\begin{equation}
	\label{eq:asymptotic_Einstein_rephrased}
	\overline{g}=r^2g=dr^2+\Hat{g}+O(r^2).
\end{equation}
As for $\overline{N}$, since we do not impose \eqref{eq:asymptotic_Einstein}, we simply identify
an open neighborhood of $\bdry N\subset\overline{N}$ and $\bdry N\times[0,\rho^*)$ by a diffeomorphism.
Then we get
\begin{equation}
	\label{eq:normalization_of_h}
	\overline{h}=\rho^2 h=d\rho^2+\Hat{h}+O(\rho).
\end{equation}
Such identifications are fixed throughout this article.
From now on, as in \eqref{eq:asymptotic_Einstein_rephrased} and \eqref{eq:normalization_of_h},
we omit both ``as $r\to 0$'' and ``as $\rho\to 0$'' in the big/small $O$ notations. 

\begin{lem}
	\label{lem:Laplacian_estimate_of_C1_function}
	Let $p\in\bdry M$ and $\mathcal{U}$ an open neighborhood of $p$ in $\overline{M}$.
	If $w\in C^1(\mathcal{U})\cap C^2(\mathring{\mathcal{U}})$,
	where $\mathring{\mathcal{U}}=\mathcal{U}\setminus\bdry M$,
	then there exists a sequence $\set{p_k}$ of points in $\mathring{\mathcal{U}}$ converging to $p$ for which
	\begin{equation*}
		(r\Delta_{\overline{g}}w)(p_k)\to 0\qquad \text{as $k\to\infty$},
	\end{equation*}
	where $\Delta_{\overline{g}}$ is the (nonpositive) Laplacian with respect to $\overline{g}$.
\end{lem}

\begin{proof}
	By shrinking $\mathcal{U}$ if necessary, we may assume that $U=\mathcal{U}\cap\bdry M$ admits a
	coordinate system $(x^1,\dots,x^m)$ and $\mathcal{U}$ is identified with
	$B(0,R)\times[0,R)\subset\mathbb{R}^m\times[0,r^*)$ by these coordinates and $r$.
	Moreover, it suffices to assume that $p$ corresponds to the origin $0\in B(0,R)$ and
	$\tensor{\Hat{g}}{_a_b}(p)=\tensor{\delta}{_a_b}$ is the identity matrix.
	For $k\ge 1$, let $q_k=(0,\dots,0,2r_k)$, where $0<r_k<R/3$ and $r_k\to 0$, and $B_k=B(q_k,r_k)$.
	Let $\nu_{\overline{g}}$ be the unit outward vector field along $\bdry B_k$ with respect to $\overline{g}$ and
	$dS_{\overline{g}}$ the volume density on $\bdry B_k$ induced by $\overline{g}$.
	Then by the divergence theorem,
	\begin{equation*}
		\frac{1}{r_k^m}\int_{B_k}(\Delta_{\overline{g}}w)dV_{\overline{g}}
		= \frac{1}{r_k^m}\int_{\bdry B_k}
		\braket{\grad_{\overline{g}}w,\nu_{\overline{g}}}_{\overline{g}}dS_{\overline{g}}.
	\end{equation*}
	Viewing the gradient $\grad_{\overline{g}}w$ of $w$ with respect $\overline{g}$ 
    as an $\mathbb{R}^m$-valued function on $\bdry B_k$,
	we decompose the integral as follows, where $\gamma=\grad_{\overline{g}}w(p)$:
	\begin{equation*}
		\frac{1}{r_k^m}\int_{\bdry B_k}
		\braket{\grad_{\overline{g}}w,\nu_{\overline{g}}}_{\overline{g}}dS_{\overline{g}}
		=\frac{1}{r_k^m}\int_{\bdry B_k}
		\braket{\gamma,\nu_{\overline{g}}}_{\overline{g}}dS_{\overline{g}}
		+\frac{1}{r_k^m}\int_{\bdry B_k}
		\braket{\grad_{\overline{g}}w-\gamma,\nu_{\overline{g}}}_{\overline{g}}
		dS_{\overline{g}}.
	\end{equation*}
	Since $\grad_{\overline{g}}w$ is continuous up to $\mathcal{U}\cap\bdry M$,
	the second term in the right-hand side tends to $0$ as $k\to\infty$.
	To compute the first term, let $B\subset\mathbb{R}^{m+1}$ be the unit ball and
	$\psi_k\colon\overline{B}\longrightarrow\overline{B}_k$ be the mapping $x\mapsto q_k+r_kx$. Then,
	$r_k^{-2}\psi_k^*\overline{g}$ converges to the Euclidean metric $g_{\mathbb{E}}$ uniformly on $\overline{B}$.
	Therefore,
	\begin{equation*}
		\begin{split}
			\frac{1}{r_k^m}\int_{\bdry B_k}\braket{\gamma,\nu_{\overline{g}}}_{\overline{g}}dS_{\overline{g}}
			&=\frac{1}{r_k^m}\int_{\partial B}
			\braket{r_k^{-1}\gamma,\nu_{\psi_k^*\overline{g}}}_{\psi_k^*\overline{g}}
			dS_{\psi_k^*\overline{g}}
			=\int_{\partial B}
			\braket{\gamma,\nu_{r_k^{-2}\psi_k^*\overline{g}}}_{r_k^{-2}\psi_k^*\overline{g}}
			dS_{r_k^{-2}\psi_k^*\overline{g}}\\
			&\to \int_{\partial B}\braket{\gamma,\nu_{g_{\mathbb{E}}}}_{g_{\mathbb{E}}}dS_{g_{\mathbb{E}}}=0.
		\end{split}
	\end{equation*}
	Hence
	\begin{equation*}
		\frac{r_k}{\Vol(B_k)}\int_{B_k}(\Delta_{\overline{g}}w)dV_{\overline{g}}\to 0,
	\end{equation*}
	which implies that we can choose $p_k\in B_k$ for each $k$
	so that $r_k(\Delta_{\overline{g}}w)(p_k)\to 0$. Since $r(p_k)<3r_k$, the lemma follows.
\end{proof}

Any coordinate neighborhood $(U;x^a)=(U;x^1,\dots,x^m)$ of $\bdry M$
gives rise to a coordinate neighborhood $(\mathcal{U};x^1,\dots,x^m,r)$,
where $\mathcal{U}=U\times[0,r_0)\subset\overline{M}$ for some $r_0<r^*$.
Such a $(\mathcal{U};x^1,\dots,x^m,r)$ will be called a \emph{normal boundary coordinate neighborhood}
of $\overline{M}$.
We also write $\mathring{\mathcal{U}}=U\times(0,r_0)$, and introduce
the following notation for functions $\varphi$ defined in $\mathring{\mathcal{U}}$:
\begin{align*}
	\varphi=O(r^l)
	&\overset{\mathrm{def}}{\Longleftrightarrow}
	\text{$r^{-l}\varphi$ is uniformly bounded as $r\to 0$ on any compact subset of $U$},\\
	\varphi=o(r^l)
	&\overset{\mathrm{def}}{\Longleftrightarrow}
	\text{$r^{-l}\varphi$ uniformly converges to $0$ as $r\to 0$ on any compact subset of $U$}.
\end{align*}
The Christoffel symbols of $g$ in $\mathring{\mathcal{U}}$ are given in terms of those of $\overline{g}$ by
\begin{subequations}
	\label{eq:Christoffel_symbols}
\begin{equation}
	\tensor{\smash{\fourIdx{g}{}{}{}{\Gamma}}}{^k_i_j}
	=\tensor{\smash{\fourIdx{\overline{g}}{}{}{}{\Gamma}}}{^k_i_j}+\tensor{X}{^k_i_j},
\end{equation}
where
\begin{equation}
	\tensor{X}{^k_i_j}=-\frac{1}{r}(\tensor{\delta}{_i^k}\tensor{\delta}{_j^\infty}
		+\tensor{\delta}{_j^k}\tensor{\delta}{_i^\infty}
		-\tensor{\overline{g}}{_i_j}\tensor{\overline{g}}{^k^l}\tensor{\delta}{_l^\infty}).
\end{equation}
\end{subequations}
Here $\tensor{\overline{g}}{^k^l}$ denotes the inverse of the metric $\overline{g}$
and the index $\infty$ denotes the $r$-direction.
The indices $i$, $j$, $k$, $l$ are running $\set{1,\dots,m,\infty}$.
The Christoffel symbols of $h$ admit the similar expression.

Let $\mathcal{U}=U\times[0,r_0)$ and $\mathcal{V}=V\times[0,\rho_0)$ be normal boundary coordinate neighborhoods
of $\overline{M}$ and $\overline{N}$, respectively,
and suppose that $u\in C^1(\mathcal{U},\mathcal{V})\cap C^2(\mathring{\mathcal{U}},\mathring{\mathcal{V}})$.
We write $u=(u^\alpha,\rho)=(u^1,\dots,u^n,\rho)$.
Then by \eqref{eq:Christoffel_symbols} and the similar expression for $\fourIdx{h}{}{}{}{\Gamma}$,
the components of the tension field $\tau=\tr_g\nabla du$ can be computed.
For later convenience, we write down the formulae of $r^{-1}\tau^\alpha$ and $r^{-1}\tau^\infty$:
\begin{subequations}
	\label{eq:tension_divided}
\begin{gather}
	\label{eq:tension_tangential_divided}
	r^{-1}\tau^\alpha
	=r\tau_{\overline{g},\overline{h}}^\alpha
	-(m-1)\frac{\partial u^\alpha}{\partial r}
	-\frac{2r}{\rho}\braket{d\rho,du^\alpha}_{\overline{g}}+O(r)+\frac{O(r^3)}{\rho},\\
	\label{eq:tension_normal_divided}
	r^{-1}\tau^\infty
	=r\tau_{\overline{g},\overline{h}}^\infty
	-(m-1)\frac{\partial\rho}{\partial r}
	-\frac{r}{\rho}\left(\abs{d\rho}^2_{\overline{g}}
	-\sum_{\alpha,\beta}\Hat{h}_{\alpha\beta}\braket{du^\alpha,du^\beta}_{\overline{g}}\right)
	+O(r)+\frac{O(r^3)}{\rho}.
\end{gather}
\end{subequations}
Here $\tau_{\overline{g},\overline{h}}$ is the tension field of $u$ with respect to
$\overline{g}$ and $\overline{h}$.
Based on these formulae, we determine the Neumann data of a harmonic mapping $u$.

\begin{lem}
	\label{lem:Neumann_data}
	Let $u\in C^1(\mathcal{U},\mathcal{V})\cap C^2(\mathring{\mathcal{U}},\mathring{\mathcal{V}})$
	satisfy $u(U)\subset V$.
	Suppose that its restriction $f=u|_U$ has nowhere vanishing differential, and
	let $\Hat{e}(f)$ be its energy density with respect to $\Hat{g}$ and $\Hat{h}$.
	Then, the tension field satisfies
	$\abs{\tau}=o(1)$ if and only if $\abs{\tau_{\overline{g},\overline{h}}}_{\overline{h}}=o(r^{-1})$ and
	\begin{equation}
		\label{eq:boundary_value_of_harmonic_map}
		\left.\frac{\partial u^\alpha}{\partial r}\right|_U=0,\quad\text{$\alpha=1,$ $\dots,$ $n$},
		\qquad\text{and}\qquad
		\left.\frac{\partial\rho}{\partial r}\right|_U=\sqrt{\frac{\Hat{e}(f)}{m}}.
	\end{equation}
\end{lem}

\begin{proof}
	Suppose that $\abs{\tau}=o(1)$, which is equivalent to $r^{-1}\tau^\alpha=o(1)$ and $r^{-1}\tau^\infty=o(1)$,
	and let $p\in U$ be arbitrary.
	Since $u^1$, $\dots$, $u^n$, and $\rho$ have continuous derivatives,
	\begin{equation*}
		\tau_{\overline{g},\overline{h}}^\infty=\Delta_{\overline{g}}\rho+O(1).
	\end{equation*}
	Then, by Lemma \ref{lem:Laplacian_estimate_of_C1_function}, there exists a sequence
	$\set{p_k}$ of points on $\mathring{\mathcal{U}}$ for which
	$p_k\to p$ and
	\begin{equation*}
		r(p_k)\cdot\tau_{\overline{g},\overline{h}}^\infty(p_k)\to 0\quad\text{as $k\to\infty$}.
	\end{equation*}
	Let $\partial\rho/\partial r=\kappa\ge 0$ on $U$. By multiplying \eqref{eq:tension_normal_divided}
	by $\rho/r$ and taking the limit of the values at $p_k$, we get
	\begin{equation}
		\label{eq:equation_for_Neumann}
		0=m\kappa(p)^2-\Hat{e}(f)(p)
		-\sum_{\alpha,\beta}\Hat{h}_{\alpha\beta}(p)
		\frac{\partial u^\alpha}{\partial r}(p)\frac{\partial u^\beta}{\partial r}(p).
	\end{equation}
	Since $\Hat{e}(f)(p)>0$,
	this in particular implies that $\kappa(p)\not=0$, which combined with
	\eqref{eq:tension_tangential_divided} shows that $(\partial u^\alpha/\partial r)(p)=0$.
	Thus we conclude from \eqref{eq:equation_for_Neumann} that
	$\kappa(p)=\sqrt{\Hat{e}(f)(p)/m}$, and this is true for any $p$.
	Now it follows from \eqref{eq:tension_tangential_divided} and \eqref{eq:tension_normal_divided} that
	$\abs{\tau_{\overline{g},\overline{h}}}_{\overline{h}}=o(r^{-1})$.
	The converse is clear.
\end{proof}

We remark that this lemma has the following consequence.

\begin{cor}
	\label{cor:properness_energy_density}
	Assume that $u\in C^1(\overline{M},\overline{N})\cap C^2(M,N)$ satisfies $u(\bdry M)\subset\bdry N$,
	the differential of $f=u|_{\bdry M}$ is nowhere vanishing, and $\abs{\tau}=o(1)$.
	Then $u$ is proper and
	the energy density $e(u)$ of $u$ with respect to $g$ and $h$ converges to $m+1$ uniformly as $r\to 0$.
\end{cor}

Next we construct approximate harmonic maps locally.
In order to do that, we will need to extend a given function on $\bdry M$ to an approximate harmonic function
with respect to $\overline{g}$ in such a way that its derivatives are controlled.
Let $\sigma_m$ be the Euclidean volume of the unit sphere in $\mathbb{R}^{m+1}$,
and recall that the Poisson kernel of the upper-half space $\set{(x,r)\in\mathbb{R}^m\times\mathbb{R}|r>0}$ is
\begin{equation*}
	K_0(x,r;x')=c_m\frac{r}{(\abs{x-x'}^2+r^2)^{(m+1)/2}},
\end{equation*}
where $c_m=2/\sigma_m$.
Our idea is to mimic this kernel function.
If $i_{\Hat{g}}(\bdry M)$ is the injectivity radius of $(\bdry M,\Hat{g})$,
then the squared distance function $d_{\Hat{g}}(x,x')^2$ is smooth in
$\set{d_{\Hat{g}}(x,x')<i_{\Hat{g}}(\bdry M)}\subset\bdry M\times\bdry M$.
Recall that the first variational formula of geodesic length implies
\begin{equation}
	\label{eq:gradient_of_distance}
	\grad d_{\Hat{g}}(x,x')=-\gamma'(0)+\gamma'(L)\in T_x\bdry M\oplus T_{x'}\bdry M
\end{equation}
in $\set{d_{\Hat{g}}(x,x')<i_{\Hat{g}}(\bdry M)}\setminus(\mathrm{diagonal})$, where
$\gamma\colon[0,L]\longrightarrow\bdry M$ is the unit-speed minimizing geodesic from $x$ to $x'$.
Hence, in the geodesic coordinates $\xi=(\xi^a)$ centered at $x\in\bdry M$,
\begin{equation}
	\label{eq:gradient_of_squared_distance}
	\grad d_{\Hat{g}}^2(0,\xi)=-2\xi+2\xi\in T_x\bdry M\oplus T_{x'}\bdry M
\end{equation}
as far as $\abs{\xi}<i_{\Hat{g}}(\bdry M)$.
Now let $\delta=i_{\Hat{g}}(\bdry M)/2$ and we choose a smooth function
$D\colon\bdry M\times\bdry M\longrightarrow[0,\infty)$ for which
\begin{equation}
	\label{eq:modified_distance_function}
	D(x,x')
	\begin{cases}
		=d_{\Hat{g}}(x,x')^2 & \text{if $d_{\Hat{g}}(x,x')<\delta$},\\
		\ge \delta^2 & \text{if $d_{\Hat{g}}(x,x')\ge\delta$}.
	\end{cases}
\end{equation}
Then note that
\begin{equation}
	\label{eq:gradient_D}
	\abs{\grad_{\Hat{g}}D(x,x')}_{\Hat{g}}^2=4D(x,x') \qquad\text{near the diagonal}
\end{equation}
and
\begin{equation}
	\label{eq:Laplacian_D}
	\Delta_{\Hat{g}}D(x,x')=2m+O(D(x,x')) \qquad\text{as $(x,x')$ tends to the diagonal},
\end{equation}
where $\grad_{\Hat{g}}$ and $\Delta_{\Hat{g}}$ apply to the $x$ variable.

We define the kernel function $K(x,r;x')$ on $(\bdry M\times(0,r^*))\times\bdry M$ by
\begin{equation}
	\label{eq:approximate_Poisson_kernel}
	K(x,r;x')=c_m\frac{r}{(D(x,x')+r^2)^{(m+1)/2}}.
\end{equation}
This is smooth everywhere in $(\bdry M\times(0,r^*))\times\bdry M$.

\begin{lem}
	\label{lem:approximate_Poisson_kernel}
	The function $K(x,r;x')$ satisfies the following:
	\begin{subequations}
	\begin{gather}
		\label{eq:integral_of_K}
		\lim_{r\to 0}\int_{\bdry M}K(x,r;x')dV_{\Hat{g}}(x')=1,\\
		\label{eq:integral_of_grad_K}
		\lim_{r\to 0}r\Abs{\int_{\bdry M}\grad_{\overline{g}}K(x,r;x')dV_{\Hat{g}}(x')}_{\overline{g}}=0,\\
		\label{eq:integral_of_Laplacian_K}
		\lim_{r\to 0}r\int_{\bdry M}\Delta_{\overline{g}}K(x,r;x')dV_{\Hat{g}}(x')=0.
	\end{gather}
	\end{subequations}
	All convergences are uniform in $x\in\bdry M$. Moreover, there exists a constant $C>0$ such that
	$\abs{r\grad_{\overline{g}}K(x,r;x')}_{\overline{g}}\le CK(x,r;x')$ and
	$\abs{\Delta_{\overline{g}}K(x,r;x')}_{\overline{g}}\le CK(x,r;x')$.
\end{lem}

\begin{proof}
	Let $\delta=i_{\Hat{g}}(\bdry M)/2$.
	For any fixed $x\in\bdry M$, let $B=B_{\Hat{g}}(x,\delta)\subset\bdry M$ be the geodesic ball
	and we decompose the first integral as
	\begin{equation*}
		\int_{\bdry M}K(x,r;x')dV_{\Hat{g}}(x')
		=\int_BK(x,r;x')dV_{\Hat{g}}(x')+\int_{\bdry M\setminus B}K(x,r;x')dV_{\Hat{g}}(x').
	\end{equation*}
	Since $x'\in\bdry M\setminus B$ implies $K(x,r;x')\le c_mr/\delta^{m+1}$,
	the integral over $\bdry M\setminus B$ is $O(r)$.
	So it suffices to estimate the integral over $B$ to show \eqref{eq:integral_of_K}.
	We express it as an integral over $\set{\abs{\xi}<\delta}\subset\mathbb{R}^n$
	by introducing the geodesic coordinates centered at $x$ on $B$. Then
	\begin{equation}
		\label{eq:change_of_variable}
		\begin{split}
			\int_BK(x,r;x')dV_{\Hat{g}}(x')
			&=c_m\int_{\abs{\xi}<\delta}\frac{r}{(\abs{\xi}^2+r^2)^{(m+1)/2}}dV_{\Hat{g}}(\xi)\\
			&=c_m\int_{\abs{\xi}<r^{-1}\delta}\frac{r}{(\abs{r\xi}^2+r^2)^{(m+1)/2}}dV_{\Hat{g}}(r\xi)\\
			&=c_m\int_{\abs{\xi}<r^{-1}\delta}\frac{1}{(\abs{\xi}^2+1)^{(m+1)/2}}\frac{dV_{\Hat{g}}(r\xi)}{r^m}.
		\end{split}
	\end{equation}
	The volume density $dV_{\Hat{g}}(r\xi)/r^m$ is equivalent to the Euclidean volume
	density $dV_{g_{\mathbb{E}}}(\xi)$ uniformly in $x$ and $r$.
	Hence, for any given $\varepsilon>0$, we can take $R>0$ so that
	\begin{equation*}
		c_m\int_{R<\abs{\xi}<r^{-1}\delta}\frac{1}{(\abs{\xi}^2+1)^{(m+1)/2}}\frac{dV_{\Hat{g}}(r\xi)}{r^m}
		<\varepsilon
		\qquad\text{and}\qquad
		c_m\int_{\abs{\xi}>R}\frac{1}{(\abs{\xi}^2+1)^{(m+1)/2}}dV_{g_{\mathbb{E}}}(\xi)
		<\varepsilon
	\end{equation*}
	for $0<r<R^{-1}\delta$. On the other hand, $dV_{\Hat{g}}(r\xi)/r^m$ converges to $dV_{g_{\mathbb{E}}}(\xi)$
	uniformly on $\set{\abs{\xi}<R}\subset\mathbb{R}^m$ and uniformly in $x$.
	Therefore the last expression in \eqref{eq:change_of_variable}
	converges as $r\to 0$ to the integral of $K_0(0,1;\xi)dV_{g_{\mathbb{E}}}(\xi)$
	over $\mathbb{R}^m$, which equals $1$, uniformly in $x$.

	The second limit \eqref{eq:integral_of_grad_K} is proved in a similar way.
	Since $\overline{g}$ and $dr^2+\Hat{g}$ are quasi-equivalent, it suffices to show that
	\begin{equation*}
		\lim_{r\to 0}r\Abs{\int_{\bdry M}\grad_{\Hat{g}}K(x,r;x')dV_{\Hat{g}}(x')}_{\Hat{g}}=0\qquad\text{and}\qquad
		\lim_{r\to 0}r\int_{\bdry M}\frac{\partial K}{\partial r}(x,r;x')dV_{\Hat{g}}(x')=0,
	\end{equation*}
	where $\grad_{\Hat{g}}$ is the gradient in the $x$ variable, and that the convergences are uniform in $x$.
	The integrands are computed as follows:
	\begin{subequations}
		\label{eq:gradient_of_approximate_Poisson_kernel}
	\begin{align}
		\grad_{\Hat{g}}K(x,r;x')&=-\frac{(m+1)c_m}{2}\frac{r\grad_{\Hat{g}}D(x,x')}{(D(x,x')+r^2)^{(m+3)/2}},\\
		\frac{\partial K}{\partial r}(x,r;x')
		&=c_m\left(\frac{1}{(D(x,x')+r^2)^{(m+1)/2}}-\frac{(m+1)r^2}{(D(x,x')+r^2)^{(m+3)/2}}\right).
	\end{align}
	\end{subequations}
	As before, it suffices to consider the integrals over $B=B_{\Hat{g}}(x,\delta)$ instead of
	those over $\bdry M$.
	We introduce the geodesic coordinates centered at $x$.
	If $\xi=(\xi^a)\in\mathbb{R}^m$ is the coordinate of $x'$, then $\grad_{\Hat{g}}D(x,x')=-2\xi$ by
	\eqref{eq:gradient_of_squared_distance}. Thus it remains to show that
	\begin{equation*}
		\lim_{r\to 0}r\int_{\abs{\xi}<\delta}\frac{r\xi^a}{(\abs{\xi}^2+r^2)^{(m+3)/2}}dV_{\Hat{g}}(\xi)=0,
		\qquad a=1,\dotsc,m
	\end{equation*}
	and
	\begin{equation*}
		\lim_{r\to 0}r\int_{\abs{\xi}<\delta}
		\left(\frac{1}{(\abs{\xi}^2+r^2)^{(m+1)/2}}-\frac{(m+1)r^2}{(\abs{\xi}^2+r^2)^{(m+3)/2}}\right)
		dV_{\Hat{g}}(\xi)=0,
	\end{equation*}
	both uniformly in $x$. These follow from the fact that the gradient of $K_0(x,r;x')$ integrates to $0$
	over $\mathbb{R}^m$.
	Also, it follows from \eqref{eq:gradient_of_approximate_Poisson_kernel} and \eqref{eq:gradient_D} that
	$\abs{r\grad_{\overline{g}}K}\le CK$ for some $C>0$.
	
	To show \eqref{eq:integral_of_Laplacian_K}, it suffices to prove that
	$\abs{\Delta_{\overline{g}}K}\le CK$ for some $C>0$, which we simply write $\Delta_{\overline{g}}K=O(K)$.
	A direct computation shows that $\abs{r^2\nabla_{\overline{g}}^2K}_{\overline{g}}=O(K)$.
	By \eqref{eq:asymptotic_Einstein_rephrased}, this together with
	$\abs{r\nabla_{\overline{g}}K}_{\overline{g}}=O(K)$ implies
	\begin{equation*}
		\Delta_{\overline{g}}K= \frac{\partial^2K}{\partial r^2}+\Delta_{\Hat{g}}K+O(K).
	\end{equation*}
	Moreover,
	\begin{gather*}
		\frac{\partial^2K}{\partial r^2}
		=c_m\left(-\frac{3(m+1)r}{(D+r^2)^{(m+3)/2}}+\frac{(m+1)(m+3)r^3}{(D+r^2)^{(m+5)/2}}\right),\\
		\Delta_{\Hat{g}}K
		=c_m\left(-\frac{m+1}{2}\frac{r\Delta_{\Hat{g}}D}{(D+r^2)^{(m+3)/2}}
		+\frac{(m+1)(m+3)}{4}\frac{r\abs{\grad_{\Hat{g}}D}_{\Hat{g}}^2}{(D+r^2)^{(m+5)/2}}\right).
	\end{gather*}
	By \eqref{eq:gradient_D} and \eqref{eq:Laplacian_D}, we conclude that $c_m^{-1}\Delta_{\overline{g}}K$ equals
	\begin{equation}
		\label{eq:Laplacian_of_K}
		\frac{3(m+1)r}{(D+r^2)^{(m+3)/2}}-\frac{(m+1)(m+3)r^3}{(D+r^2)^{(m+5)/2}}
		+\frac{m(m+1)r}{(D+r^2)^{(m+3)/2}}-\frac{(m+1)(m+3)rD}{(D+r^2)^{(m+5)/2}}
	\end{equation}
	modulo $O(K)$, and \eqref{eq:Laplacian_of_K} actually vanishes.
\end{proof}

\begin{lem}
	\label{lem:approximate_harmonic_extension_of_functions}
	Let $\varphi\in C^0(\bdry M)$. We define the function $w\in C^\infty(\bdry M\times(0,r^*))$ by
	\begin{equation*}
		w(x,r)=\int_{\bdry M}K(x,r;x')\varphi(x')dV_{\Hat{g}}(x').
	\end{equation*}
	Then the following holds:
	\begin{enumerate}
		\item $w$ is continuously extended to $\bdry M\times[0,r^*)$ and $w|_{\bdry M}=\varphi$;
		\item $r\abs{\grad_{\overline{g}}w}_{\overline{g}}=o(1)$; and
		\item $r\Delta_{\overline{g}}w=o(1)$.
		\item If $\varphi\in C^1(\bdry M)$, then $w\in C^1(\bdry M\times[0,r^*))$ and
			$r\abs{\nabla_{\overline{g}}^2w}_{\overline{g}}=o(1)$.
	\end{enumerate}
\end{lem}

\begin{proof}
	To show (i), by \eqref{eq:integral_of_K} it suffices to prove that
	\begin{equation}
		\label{eq:limit_for_boundary_value_of_approximate_Poisson_integral}
		\lim_{r\to 0}\int_{\bdry M}K(x,r;x')\abs{\varphi(x')-\varphi(x)}dV_{\Hat{g}}(x')=0,
		\qquad\text{uniformly in $x$}.
	\end{equation}
	This follows by a standard argument (see the proof of \cite[Theorem 2.6]{GilbargTrudinger}).
	By Lemma \ref{lem:approximate_Poisson_kernel}, (ii) and (iii) also follow from
	\eqref{eq:limit_for_boundary_value_of_approximate_Poisson_integral}.
	To show (iv), let $X$ be a vector field on $\bdry M$ and consider
	\begin{equation*}
		Xw(x)=\int_{\bdry M}XK(x,r;x')\varphi(x')dV_{\Hat{g}}(x').
	\end{equation*}
	In order to prove that $Xw\in C^0(\bdry M\times[0,r^*))$,
	it suffices to check that the integral over $B=B_{\Hat{g}}(x,\delta)$ converges to a continuous
	function in $x$ as $r\to 0$ uniformly.
	If $x'\in B$ and $\gamma$ the unit-speed minimizing geodesic from $x$ to $x'$,
	then by \eqref{eq:gradient_of_distance} and \eqref{eq:approximate_Poisson_kernel},
	\begin{equation*}
		XK(x,r;x')=-X'K(x,r;x'),
	\end{equation*}
	where $X'(x')$ is the parallel translation of $X(x)$ along $\gamma$ and is applied to the $x'$ variable.
	Therefore, if $\nu_{\Hat{g}}$ denotes the outward unit normal vector field along $\bdry B$,
	\begin{equation*}
		\begin{split}
			&\int_BXK(x,r;x')\varphi(x')dV_{\Hat{g}}(x')\\
			&=-\int_BX'K(x,r;x')\varphi(x')dV_{\Hat{g}}(x')\\
			&=\int_BK(x,r;x')\mathcal{L}_{X'}(\varphi\,dV_{\Hat{g}})(x')
			-\int_{\bdry B}K(x,r;x')\varphi(x')\braket{X',\nu_{\Hat{g}}}_{\Hat{g}}dS_{\Hat{g}}(x')
		\end{split}
	\end{equation*}
	and the last expression is continuous up to the boundary by (i).
	Thus all the components of $r\nabla^2_{\overline{g}}w$ except
	$r\nabla^2_{\overline{g}}w(\partial_r,\partial_r)$ are $o(1)$ by (ii).
	Finally, (iii) implies that $r\nabla_{\overline{g}}^2w(\partial_r,\partial_r)$ is also $o(1)$.
\end{proof}

\begin{prop}
	\label{prop:local_approximation}
	Suppose $f\in C^1(\bdry M,\bdry N)$ has nowhere vanishing differential.
	Let $\mathcal{U}=U\times[0,r_0)$, $\mathcal{V}=V\times[0,\rho_0)$ be normal boundary coordinate
	neighborhoods of $\overline{M}$, $\overline{N}$, respectively, for which $f(\overline{U})\subset V$.
	Then there exists
	$v\in C^1(\mathcal{U},\mathcal{V})\cap C^\infty(\mathring{\mathcal{U}},\mathring{\mathcal{V}})$
	such that $v|_U=f|_U$ and the norm $\abs{\tau}$ of the tension field of $v$ satisfies $\abs{\tau}=o(1)$.
\end{prop}

\begin{proof}
	Let $f=(f^1,\dots,f^n)$ on $U$ and $v=(v^1,\dots,v^n,\rho)$ on $\mathcal{U}$,
	the latter of which is to be determined.
	Then $v^\alpha|_U$ has to be $f^\alpha|_U$, and from Lemma \ref{lem:Neumann_data} it is also necessary that
	$(\partial v^\alpha/\partial r)|_U=0$ and $(\partial\rho/\partial r)|_U=\sqrt{\Hat{e}(f)/m}$.
	We extend the functions $f^\alpha$ to $\bdry M$ continuously differentiably,
	and define $v^1$, $\dotsc$, $v^n$, $\rho$ near $\mathcal{U}\cap\bdry M$ by
	\begin{equation}
		\label{eq:local_approx_harmonic}
		v^\alpha=w^\alpha-r\frac{\partial w^\alpha}{\partial r}\qquad\text{and}\qquad
		\rho=rw,
	\end{equation}
	where
	\begin{equation*}
		w^\alpha(x)=\int_{\bdry M}K(x,r;x')f^\alpha(x')dV_{\Hat{g}}(x')\qquad\text{and}\qquad
		w(x)=\int_{\bdry M}K(x,r;x')\sqrt{\frac{\Hat{e}(f)(x')}{m}}dV_{\Hat{g}}(x').
	\end{equation*}
	We extend these functions to $\mathcal{U}$ so that they determine a map
	$v\in C^0(\mathcal{U},\mathcal{V})\cap C^\infty(\mathring{\mathcal{U}},\mathring{\mathcal{V}})$.
	Differentiating \eqref{eq:local_approx_harmonic}, near $\mathcal{U}\cap\bdry M$ we get
	\begin{equation*}
		\frac{\partial v^\alpha}{\partial x^a}
		=\frac{\partial w^\alpha}{\partial x^a}-r\frac{\partial^2 w^\alpha}{\partial x^a\partial r},\qquad
		\frac{\partial v^\alpha}{\partial r}
		=-r\frac{\partial^2 w^\alpha}{\partial r^2},\qquad
		\frac{\partial\rho}{\partial x^a}=r\frac{\partial w}{\partial x^a},\qquad
		\frac{\partial\rho}{\partial r}=w+r\frac{\partial w}{\partial r}.
	\end{equation*}
	By Lemma \ref{lem:approximate_harmonic_extension_of_functions}, these are continuous
	up to the boundary, which implies that $v$ is actually in $C^1(\mathcal{U},\mathcal{V})$.
	Moreover, $\abs{\tau}=o(1)$ by Lemmas \ref{lem:Neumann_data} and \ref{lem:approximate_harmonic_extension_of_functions}.
\end{proof}

We patch up the local approximate harmonic maps above to get the following result.

\begin{prop}
	\label{prop:global_approximation}
	Suppose $u_0\in C^1(\overline{M},\overline{N})$ satisfies $u_0(\bdry M)\subset\bdry N$ and
	$f=u_0|_{\bdry M}$ has nowhere vanishing differential.
	Then, there exists $v\in C^1(\overline{M},\overline{N})\cap C^\infty(M,N)$ in the same
	relative homotopy class for which $\abs{\tau}=o(1)$.
\end{prop}

\begin{proof}
	We may assume that $u_0\in C^1(\overline{M},\overline{N})\cap C^\infty(M,N)$ by mollification.
	We take finite sets of normal boundary coordinate neighborhoods $\set{\mathcal{U}_i}_{i=1}^l$
	and $\set{\mathcal{V}_i}_{i=1}^l$ of $\overline{M}$ and $\overline{N}$, respectively,
	so that the following conditions are satisfied:
	\begin{itemize}
		\item $\set{U_i}_{i=1}^l$ covers $\bdry M$, where $U_i=\mathcal{U}_i\cap\bdry M$;
		\item each $V_i=\mathcal{V}_i\cap\bdry N$ is mapped by the coordinates $(y^1,\dots,y^n)$ onto a
			convex open subset of $\mathbb{R}^n$;
		\item each pair of $\mathcal{U}_i$ and $\mathcal{V}_i$ satisfies the assumption of Proposition
			\ref{prop:local_approximation}; hence we can take a map
			$v_i\in C^1(\mathcal{U}_i,\mathcal{V}_i)\cap
			C^\infty(\mathring{\mathcal{U}}_i,\mathring{\mathcal{V}}_i)$ such that
			$v_i|_{U_i}=f|_{U_i}$ and $\abs{\tau(v_i)}=o(1)$;
		\item $u_0(\mathcal{U}_i)\subset\mathcal{V}_i$.
	\end{itemize}
	We take a relatively compact subset $\mathcal{U}'_i=U'_i\times[0,r'_i)$ of each $\mathcal{U}_i$
	so that $\set{U'_i}$ still covers $\bdry M$, and introduce the following notation:
	\begin{equation*}
		\tilde{\mathcal{U}}_i=\bigcup_{j=1}^i\mathcal{U}'_j.
	\end{equation*}
	Now we shall inductively define $\tilde{v}_i\in C^1(\overline{M},\overline{N})\cap C^\infty(M,N)$ so that
	it lies in the relative homotopy class of $u_0$, $\tilde{v}_i(\mathcal{U}'_j)\subset\mathcal{V}_j$
	for $j=i+1,$ $\dots,$ $l$, and $\abs{\tau(\tilde{v}_i)}=o(1)$ in $\tilde{\mathcal{U}}_i$.
	In this process we allow ourselves to shrink $\mathcal{U}'_i$ by making $r'_k$ smaller.
	If it is once done, then $v=\tilde{v}_l$ becomes the desired map.
	We put $\tilde{v}_0=u_0$, which is regarded as trivially satisfying the requirement.
	Suppose $\tilde{v}_{i-1}$ is constructed.
	Let $\psi\colon\overline{M}\longrightarrow[0,1]$ be a smooth function supported in $\mathcal{U}_i$
	such that $\psi=1$ in $\mathcal{U}'_i$. Then we define $\tilde{v}_i$ first in $\mathcal{U}_i$ by
	\begin{equation*}
		\tilde{v}_i^\alpha=(1-\psi)\tilde{v}_{i-1}^\alpha+\psi v_i^\alpha
		\qquad\text{and}\qquad
		\tilde{v}_i^\infty=(1-\psi)\tilde{v}_{i-1}^\infty+\psi v_i^\infty
	\end{equation*}
	using the coordinates in $\mathcal{V}_i$, which makes sense by the convexity of the image of $V_i$ in
	$\mathbb{R}^n$.
	We extend it to $\overline{M}$ by setting
	$\tilde{v}_i=\tilde{v}_{i-1}$ in $\overline{M}\setminus\mathcal{U}_i$.
	Then $\tilde{v}_i\in C^1(\overline{M},\overline{N})\cap C^\infty(M,N)$ stays in the same relative homotopy
	class, and since $\tilde{v}_i|_{\bdry M}=f$, we can replace $r'_j$, $j=i+1$, $\dots$, $l$ with smaller numbers
	so that $\tilde{v}_i(\mathcal{U}'_j)\subset\mathcal{V}_j$ for those $j$.
	As for the tension field, as $\tilde{v}_{i-1}$ and $v_i$ are $C^1$ up to $\mathcal{U}_i\cap\bdry M$,
	it is clear that
	\begin{equation*}
		\tau_{\overline{g},\overline{h}}(\tilde{v}_i)
		=(1-\psi)\tau_{\overline{g},\overline{h}}(\tilde{v}_{i-1})+\psi\tau_{\overline{g},\overline{h}}(v_i)+O(1)
	\end{equation*}
	and hence $\abs{\tau_{\overline{g},\overline{h}}(\tilde{v}_i)}_{\overline{h}}=o(r^{-1})$
	by Lemma \ref{lem:Neumann_data}. Moreover, since $\tilde{v}_{i-1}$ and $v_i$ both satisfy
	\eqref{eq:boundary_value_of_harmonic_map}, so does $\tilde{v}_i$.
	Therefore by Lemma \ref{lem:Neumann_data} again, we obtain $\abs{\tau(\tilde{v}_i)}=o(1)$ in $\mathcal{U}_i$,
	and hence in $\tilde{\mathcal{U}}_i$.
\end{proof}

\section{Existence and uniqueness}
\label{sec:existence}

For $0<\delta<r^*$, we set $\mathcal{B}_\delta=\mathcal{B}^M_\delta=M\setminus\set{0<r\le\delta}$.
The subsets $\mathcal{B}^N_\kappa$ of $N$ are similarly defined for $0<\kappa<\rho^*$.
Given $u_0\in C^1(\overline{M},\overline{N})$ such that $u_0(\bdry M)\subset\bdry N$,
let $v\in C^1(\overline{M},\overline{N})\cap C^\infty(M,N)$ be an approximate solution
to the harmonic map equation constructed in Proposition \ref{prop:global_approximation}.
Then for each $\delta\in(0,r^*)$, by Hamilton's work~\cite{Hamilton}, there exists a harmonic map 
$u_\delta\in C^\infty(\overline{\mathcal{B}_\delta},N)$ satisfying
$u_\delta|_{\bdry\mathcal{B}_\delta}=v|_{\bdry\mathcal{B}_\delta}$
that is relatively homotopic to $v|_{\overline{\mathcal{B}}_\delta}$.
To show this, it suffices to verify that $\bdry\mathcal{B}^N_\kappa$ is convex in $N$,
and it is easily observed from the formula for $\fourIdx{h}{}{}{}{\Gamma}$ that is similar to
\eqref{eq:Christoffel_symbols}.

Let $d_\delta\colon\overline{\mathcal{B}_\delta}\longrightarrow[0,\infty)$ be the function defined by
\begin{equation*}
	d_\delta(p)=d(u_\delta(p),v(p)),\qquad p\in\overline{\mathcal{B}_\delta},
\end{equation*}
where $d$ is the distance function of $h$.
In Proposition \ref{prop:uniform_boundedness_of_exhaustion}, we will show that $d_\delta$ is uniformly bounded
for small $\delta>0$.
Moreover, we consider the distance of $u_\delta$ and $v$ ``measured on the universal cover''
following Schoen and Yau~\cite{SchoenYau}.
Let $\tilde{M}$, $\tilde{N}$ be the universal covers of $M$, $N$.
They are equipped with $\tilde{g}=\varpi_M^*g$ and $\tilde{h}=\varpi_N^*h$,
where $\varpi_M$ and $\varpi_N$ are the standard projections.
Let $\tilde{v}\colon\tilde{M}\longrightarrow\tilde{N}$ a lift of $v\circ\varpi_M$.
We also take the lift
$\tilde{u}_\delta\colon\varpi_M^{-1}(\overline{\mathcal{B}_\delta})\longrightarrow\tilde{N}$
of $u_\delta\circ\varpi_M$ in such a way that $\tilde{u}_\delta$ is relatively homotopic to
$\tilde{v}|_{\varpi_M^{-1}(\overline{\mathcal{B}}_\delta)}$.
By using the distance function $\tilde{d}$ of $\tilde{h}$, we set
\begin{equation*}
	\tilde{d}_\delta(\tilde{p})=\tilde{d}(\tilde{u}_\delta(\tilde{p}),\tilde{v}(\tilde{p})),\qquad
	\tilde{p}\in\varpi_M^{-1}(\overline{\mathcal{B}_\delta}).
\end{equation*}
For each $\alpha\in\pi_1(M)=\pi_1(\overline{\mathcal{B}_\delta})$, there is $\beta\in\pi_1(N)$ for which
\begin{equation*}
	\tilde{u}_\delta(\alpha\cdot\tilde{p})=\beta\cdot\tilde{u}_\delta(\tilde{p})\qquad\text{and}\qquad
	\tilde{v}(\alpha\cdot\tilde{p})=\beta\cdot\tilde{v}(\tilde{p})
	\qquad\text{for any $\tilde{p}\in\varpi_M^{-1}(\overline{\mathcal{B}_\delta})$}.
\end{equation*}
In fact, $\beta$ can be chosen to be either the homotopy class $(u_\delta)_*\alpha$ or
$v_*\alpha$, which actually coincide.
Since $\pi_1(N)$ acts as an isometry on $\tilde{N}$, $\tilde{d}_\delta$ descends to a function on
$\overline{\mathcal{B}_\delta}$.
It is clear from the definition that $d_\delta\le\tilde{d}_\delta$.

As $\tilde{N}$ is an Hadamard manifold,
$\tilde{d}\colon\tilde{N}\times\tilde{N}\longrightarrow[0,\infty)$ is smooth away from the diagonal.
Define $\psi\colon\varpi_M^{-1}(\overline{\mathcal{B}_\delta})\longrightarrow\tilde{N}\times\tilde{N}$ by
\begin{equation*}
	\psi(\tilde{p})=(\tilde{u}_\delta(\tilde{p}),\tilde{v}(\tilde{p})),\qquad
	\tilde{p}\in\varpi_M^{-1}(\overline{\mathcal{B}_\delta}).
\end{equation*}
Let $\Delta_g$ be the (nonpositive) Laplacian of $g$.
If restricted to $\set{\tilde{d}_\delta>0}$ and lifted to the inverse image by $\varpi_M$,
$\Delta_g\tilde{d}_\delta$ is computed as follows, where $\nabla$ and $\grad$ are taken with respect to
the product metric:
\begin{equation*}
	\label{eq:Laplacian_of_difference_of_mappings}
	\Delta_g\tilde{d}_\delta
	=\tr_{\tilde{g}}\psi^*\nabla^2\tilde{d}
	+\braket{(\grad\tilde{d})\circ\psi,\tau(u_\delta)+\tau(v)}.
\end{equation*}
By the first variational formula and the fact that $u_\delta$ is a harmonic map, we can conclude that
\begin{equation}
	\label{eq:Laplacian_of_difference_of_mappings_simplified}
	\Delta_g\tilde{d}_\delta \ge \tr_{\tilde{g}}\psi^*\nabla^2\tilde{d}-\abs{\tau(v)}.
\end{equation}

We follow the argument of J\"ager and Kaul~\cite{JagerKaul} to get an estimate of the Hessian of $\tilde{d}$.
Let $q_1$, $q_2\in\tilde{N}$ be different points and $\gamma\colon[0,L]\longrightarrow\tilde{N}$
the unit-speed geodesic from $q_1$ to $q_2$, and $v_1\in T_{q_1}\tilde{N}$, $v_2\in T_{q_2}\tilde{N}$.
Take the Jacobi field $X$ along $\gamma$ such that $X(0)=v_1^{\mathrm{nor}}$ and $X(L)=v_2^{\mathrm{nor}}$,
where $v_1^{\mathrm{nor}}$ and $v_2^{\mathrm{nor}}$ are the normal parts with respect to $\gamma'$.
Then $X$ itself is normal to $\gamma'$. A standard argument (see \cite[Equation (3.4)]{JagerKaul}) shows that
\begin{equation}
	\label{eq:Hessian_of_distance_function}
	(\nabla^2\tilde{d})(v,v)=\braket{X,X'}(L)-\braket{X,X'}(0),
	\qquad v=v_1+v_2\in T_{q_1}\tilde{N}\oplus T_{q_2}\tilde{N},
\end{equation}
where the primes denotes the covariant differentiation by $\gamma'$.
We shall apply to the right-hand side a version of Rauch's comparison theorem.
Suppose that $\mu\colon[0,L]\longrightarrow\mathbb{R}$ satisfies
\begin{equation*}
	\mu(t)\ge\sup_\sigma K_{\tilde{h}}(\sigma),
\end{equation*}
where $K_{\tilde{h}}(\sigma)$ is the sectional curvature of the plane $\sigma\subset T_{\gamma(t)}\tilde{N}$ and
$\sigma$ runs all the planes containing $\gamma'(t)$.
Let $Y$ be a Jacobi field along $\gamma$ that is normal to $\gamma'$ such that $Y(0)=0$.
Then if the solution $s(t)$ of the equation $s''+\mu s=0$, $s(0)=0$, $s'(0)=1$
satisfies $s(t)>0$ for $0<t\le L$, we obtain (see~\cite[A2]{Karcher})
\begin{equation*}
	\frac{\abs{Y(t_1)}}{s(t_1)}\le\frac{\abs{Y(t_2)}}{s(t_2)}\qquad\text{for $0<t_1\le t_2\le L$}.
\end{equation*}
Suppose moreover that $s'(t)>0$ for $t>0$. Then
\begin{equation*}
	\frac{\abs{Y}'(t_1)}{s'(t_1)}\le\frac{\abs{Y(t_2)}}{s(t_2)},\qquad
	\frac{\abs{Y(t_1)}}{s(t_1)}\le\frac{\abs{Y}'(t_2)}{s'(t_2)}\qquad\text{for $0<t_1\le t_2\le L$}.
\end{equation*}
Thus we obtain
\begin{equation}
	\label{eq:comparison_inequality}
	\abs{Y}'(0)\le\frac{1}{s(t)}\abs{Y(t)},\qquad
	\braket{Y(t),Y'(t)}\ge\frac{s'(t)}{s(t)}\abs{Y(t)}^2\qquad\text{for $0<t\le L$}.
\end{equation}

Li and Tam~\cite{LiTam3} used these inequalities in the following way.
If $(N,h)$ has negative sectional curvature with upper bound $-\kappa^2$, where $\kappa>0$, one has
\begin{equation*}
	\abs{Y'(0)}\le\frac{\kappa}{\sinh\kappa L}\abs{Y(L)},\qquad
	\braket{Y(L),Y'(L)}\ge\frac{\kappa\cosh\kappa L}{\sinh\kappa L}\abs{Y(L)}^2.
\end{equation*}
This leads to the estimate (see~\cite[3.8.~Lemma]{JagerKaul})
\begin{equation}
	\label{eq:Laplacian_estimate_for_negative_curvature}
	\Delta_g\tilde{d}_\delta
	\ge\frac{\kappa(\cosh\kappa L-1)}{\sinh\kappa L}e(v)-\abs{\tau(v)}.
\end{equation}
If $\Delta_g\tilde{d}_\delta$ is understood as a distribution, then this is valid
not only in $\set{\tilde{d}_\delta>0}$ but also in $\mathcal{B}_\delta$ (see~\cite[Section 2]{DingWang}).
Under our assumption that $(N,h)$ has nonpositive sectional curvature,
one obtains $\tr_{\tilde{g}}\psi^*\nabla^2\tilde{d}\ge 0$ and hence $\Delta_g\tilde{d}_\delta\ge -\abs{\tau(v)}$
in a similar way, but we need a more subtle analysis.

\begin{lem}
	\label{lem:comparison}
	There exist constants $c>0$ and $l>0$ for which, if $\tilde{p}\in\tilde{M}$ satisfies $d_\delta(p)\ge l$,
	where $p=\varpi_M(\tilde{p})$, then
	\begin{equation}
		\label{eq:hessian_estimate}
		(\tr_{\tilde{g}}\psi^*\nabla^2\tilde{d})(\tilde{p})\ge ce(v)(p).
	\end{equation}
\end{lem}

\begin{proof}
	Take a compact subset $\mathcal{K}\subset N$ so that the sectional curvature satisfies $K_h\le -1/2$
	in $N\setminus\mathcal{K}$.
	Let $l=\diam\mathcal{K}+4$. Then, since $d(u_\delta(p),\mathcal{K})\ge 2$ or $d(v(p),\mathcal{K})\ge 2$
	is satisfied,
	$\tilde{u}_\delta(\tilde{p})$ and $\tilde{v}(\tilde{p})$ can be connected by a unit-speed geodesic
	$\gamma\colon[0,L]\longrightarrow\tilde{N}$ for which $K_{\tilde{h}}\le -1/2$ along $\gamma|_{[L-2,L]}$.
	By the argument above the lemma, it suffices to take $c>0$ so that
	\begin{equation*}
		\abs{Y'(0)}\le c\abs{Y(L)}\qquad\text{and}\qquad
		\braket{Y(L),Y'(L)}\ge 2c\abs{Y(L)}^2
	\end{equation*}
	for any Jacobi field $Y$ along $\gamma$ with $Y(0)=0$ that is perpendicular to $\gamma'$. Let
	\begin{equation*}
		\mu(t)=\sup_\sigma K_{\tilde{h}}(\sigma),
	\end{equation*}
	where $\sigma$ runs all the planes of $T_{\gamma(t)}\tilde{N}$ containing $\gamma'(t)$, and $s(t)$
	the solution of $s''+\mu s=0$, $s(0)=0$, $s'(0)=1$. The differential equation implies that, if $q=s'/s$,
	\begin{equation*}
		q'=-\mu-q^2.
	\end{equation*}
	Since $\mu(t)\le 0$, $q$ cannot change sign, so $q(t)>0$.
	Moreover, during $L-2\le t\le L$ we have $q'\ge 1/2-q^2$, from which one obtains $q(L)\ge 1/2$.
	On the other hand, we have $s'(t)\ge 1$ for $0\le t\le L$, so $s(L)\ge L\ge l\ge 4$.
	Therefore, from \eqref{eq:comparison_inequality}, we can take $c=1/4$.
\end{proof}

To establish a uniform bound of $d_\delta$, we also need the following.

\begin{lem}
	\label{lem:barrier}
	There exists a function $\varphi\in C^2(M)$ satisfying
	$\varphi>0$, $\Delta_g \varphi < 0$ in $M$ and $\varphi\to 0$ uniformly as $r\to 0$.
\end{lem}

\begin{proof}
	We first prove the existence of a bounded function $v\in C^2(M)$
	for which $\psi=\log r+v$ satisfies $\Delta_g \psi = - m$ and $\abs{\nabla\psi}$ is bounded
	(although it is actually overkilling to show the lemma).
	When $\overline{g}$ is smooth up to the boundary, it was shown by
	Fefferman and Graham~\cite[Theorem 4.1]{FeffermanGraham}
	based on the analysis of Mazzeo--Melrose~\cite{MazzeoMelrose} and Graham--Zworski~\cite{GrahamZworski}
	(the assumption made in \cite{FeffermanGraham} that the metric $g$ is approximately Einstein is
	irrelevant to this result).
	In the general case, we use the Fredholm theorem of Biquard~\cite[Proposition I.3.5]{Biquard} and
	Lee~\cite[Theorem C]{Lee}.
	Let $C_1^{k,\alpha}(M) = rC^{k,\alpha}(M)$, where $C^{k,\alpha}(M)$ is the space of $C^k$ functions
	whose intrinsic $C^{k,\alpha}$ norm is finite.
	Then $- \Delta_g \colon C_1^{2,\alpha}(M)\longrightarrow C_1^{0,\alpha}(M)$
	is a Fredholm operator of index zero and its kernel is the same as the $L^2$-kernel of $- \Delta_g$.
	But the $L^2$-kernel is obviously trivial, so this is actually an isomorphism.
	Since $m + \Delta_g \log r \in C^{0,\alpha}_1(M)$ by \eqref{eq:Christoffel_symbols},
	there exists $v\in C_1^{2,\alpha}(M)$ such that $- \Delta_g v = m  + \Delta_g \log r$.
	Moreover $\abs{\nabla\psi}$ is bounded because $v\in C_1^{2,\alpha}(M)$ implies
	$\abs{\nabla v}\in C_1^{1,\alpha}(M)$.

	Now we take such $\psi=\log r+v$ and set $\varphi=e^{\varepsilon\psi}=r^\varepsilon e^{\varepsilon v}$.
	Then clearly $\varphi$ uniformly tends to $0$ at the boundary, and
	\begin{equation}
		\label{eq:Laplacian_of_barrier}
		\Delta_g \varphi = \varepsilon e^{\varepsilon\psi}(\Delta_g \psi + \varepsilon\abs{\nabla\psi}^2) 
        = \varepsilon e^{\varepsilon\psi}(- m + \varepsilon\abs{\nabla\psi}^2).
	\end{equation}
	If we take sufficiently small $\varepsilon>0$, then $\Delta_g \varphi < 0$ holds everywhere.
\end{proof}

\begin{prop}
	\label{prop:uniform_boundedness_of_exhaustion}
	There exist constants $\delta_0>0$ and $C>0$ such that, for any $\delta\in (0,\delta_0)$,
	\begin{equation}
		d_\delta \leq C\qquad \text{in $\overline{\mathcal{B}_\delta}$}.
	\end{equation}
\end{prop}

\begin{proof}
	For $c>0$ in Lemma \ref{lem:comparison},
	by Corollary \ref{cor:properness_energy_density} we can take $\delta_0 > 0$ so that
	\begin{equation*}
		\abs{\tau(v)(p)}<c\inf_{M\setminus\mathcal{B}_{\delta_0}}e(v)
		\qquad\text{for $p\in M\setminus\mathcal{B}_{\delta_0}$}.
	\end{equation*}
	Then by \eqref{eq:Laplacian_of_difference_of_mappings_simplified} and \eqref{eq:hessian_estimate},
	$\Delta_g\tilde{d}_\delta(p) > 0$ when $p \in \overline{\mathcal{B}_\delta} \setminus \mathcal{B}_{\delta_0}$
	and $d_\delta(p)\ge l$.
	Let $\varphi$ be the function in Lemma \ref{lem:barrier}, and we take $C_1>0$ for which
	$- C_1 \Delta_g \varphi > \abs{\tau(v)}$ in $\mathcal{B}_{\delta_0}$ so that
	\begin{equation}
		\label{eq:Laplacian_of_distance_minus_barrier}
		\Delta_g(\tilde{d}_\delta-C_1\varphi) > 0\qquad \text{in $\mathcal{B}_{\delta_0}$}.
	\end{equation}
	Now let $p_0\in\overline{\mathcal{B}_\delta}$ be a point at which $\tilde{d}_\delta-C_1\varphi$
	attains its maximum in $\overline{\mathcal{B}_\delta}$.
	If $p_0\in\bdry\mathcal{B}_\delta$, then since $\tilde{d}_\delta=0$ on $\bdry\mathcal{B}_\delta$,
	it follows that $d_\delta\le\tilde{d}_\delta\le C_1\sup_M\varphi$ in $\overline{\mathcal{B}_\delta}$.
	If $p_0\not\in\bdry\mathcal{B}_\delta$, then $p_0$ has to be in
	$\overline{\mathcal{B}_\delta} \setminus \mathcal{B}_{\delta_0}$
	by \eqref{eq:Laplacian_of_distance_minus_barrier}, and hence $d_\delta(p_0)<l$.
	Therefore, $d_\delta\le l+C_1\sup_M\varphi$ in $\overline{\mathcal{B}_\delta}$.
\end{proof}

We finish the proof of the main theorem.
Note that, since $(N, h)$ is asymptotically hyperbolic,
the pointwise injectivity radius $i_h\colon N\longrightarrow (0,\infty)$
uniformly diverges to infinity at $\bdry N$.

\begin{proof}[Proof of Theorem \ref{thm:existence}]
	Recall from Corollary \ref{cor:properness_energy_density} that $v$ has bounded energy density.
	With this and Proposition \ref{prop:uniform_boundedness_of_exhaustion},
	we can show using Cheng's interior gradient estimate~\cite{Cheng} (cf.\,\cite[Theorem\,4.8.1]{Jost})
	that there exist $C>0$ and $\delta_0 >0$ such that, for any $\delta\in(0,\delta_0)$,
	\begin{equation*}
		e(u_\delta) \leq C\qquad\text{in $\overline{\mathcal{B}_{2\delta}}$}.
	\end{equation*}
	Then by the elliptic $W^{2, p}$ estimates and the Sobolev imbedding theorem,
	for each $\alpha\in(0,1)$ there exists $C_{\alpha} > 0$ such that 
	\begin{equation*}
		\norm{du_\delta}_{0,\alpha;\overline{\mathcal{B}_{3\delta}}} \leq C_{\alpha} 
	\end{equation*}
	for any $\delta\in(0,\delta_0)$. Hence, the Schauder interior estimate implies
	\begin{equation*}
		\norm{du_\delta}_{1,\alpha;\overline{\mathcal{B}_{4\delta}}} \leq C'_{\alpha} 
	\end{equation*}
	for some $C'_{\alpha} > 0$ and for any $\delta\in(0,\delta_0)$.
	Thus the harmonic maps $u_\delta$ subconverge to a harmonic map $u\in C^\infty(M,N)$ in the $C^2$ topology
	on every compact subset of $M$.
	Clearly, $d_0=d(u,v)$ is also bounded in $M$.
	Therefore $u$ lies in $C^0(\overline{M},\overline{N})$, $u|_{\bdry M}=f$, and
	$\rho\circ u$ as well as $\rho\circ v$ is quasi-equivalent to $r$.

	Next we prove that $u$ is relatively homotopic to $u_0$. If we extend $u_\delta$ to $\overline{M}$ by setting
	$u_\delta|_{\overline{M}\setminus\overline{\mathcal{B}}_\delta}
	=v|_{\overline{M}\setminus\overline{\mathcal{B}}_\delta}$,
	$u_\delta$ is relatively homotopic to $u_0$,
	so it suffices to show that $u$ is relatively homotopic to some $u_\delta$.
	Let $\delta_0$ and $C$ be the constants in Lemma \ref{prop:uniform_boundedness_of_exhaustion}.
	Then if $\delta_1\in(0,\delta_0)$ is sufficiently small, $p\in M\setminus\mathcal{B}_{\delta_1}$ implies
	$i_h(u(p))>2C$.
	Since $u_\delta$ subconverges to $u$ on $\mathcal{B}_{\delta_1}$,
	we can take $\delta\in(0,\delta_1)$ such that
	\begin{equation*}
		d(u(p),u_\delta(p))<i_h(N),\qquad p\in\mathcal{B}_{\delta_1}.
	\end{equation*}
	We have $L_p=d(u(p),u_\delta(p))<i_h(u(p))$ for any $p\in M$.
	Hence there exists a unique unit-speed minimizing geodesic $\gamma_p\colon[0,L_p]\longrightarrow N$
	from $u(p)$ to $u_\delta(p)$.
	We define $\Phi\colon M\times[0,1]\longrightarrow N$ by
	\begin{equation*}
		\Phi(p,t)=\gamma_p(tL_p).
	\end{equation*}
	Then $\Phi$ is continuous in $p$ and $t$. Moreover, it continuously extends to a map
	$\overline{M}\times[0,1]\longrightarrow\overline{N}$ so that
	$\Phi(p,t)=f(p)$ for $p\in\bdry M$ and $t\in[0,1]$.
	Therefore, $u$ and $u_\delta$ are relatively homotopic.

	We now prove that $d_0=d(u,v)$ tends to $0$ uniformly at $\bdry M$. In fact, we show that
	\begin{equation*}
		\tilde{d}_0=\tilde{d}(\tilde{u},\tilde{v})\to 0\qquad\text{uniformly at $\bdry M$},
	\end{equation*}
	where $\tilde{u}$, $\tilde{v}\colon\tilde{\overline{M}}\longrightarrow\tilde{\overline{N}}$ are
	relatively homotopic lifts of $u\circ\varpi_M$, $v\circ\varpi_M$.
	Note that, since $u$ and $v$ are proper and $i_h(q)\to\infty$ as $q\to\bdry N$,
	Proposition \ref{prop:uniform_boundedness_of_exhaustion} implies that
	$u(p)$ and $v(p)$ are connected by a unique minimizing geodesic $\gamma_0$ along which $K_h\le -1/4$ for
	$p\in M\setminus\overline{\mathcal{B}_\delta}$ as long as $\delta>0$ is sufficiently small.
	Because of the asymptotic hyperbolicity, we can take a smaller $\delta>0$ if necessary so that
	$\gamma_0$ actually lifts to a geodesic $\gamma$ connecting $\tilde{u}(\tilde{p})$ and
	$\tilde{v}(\tilde{p})$ for some $\tilde{p}\in\varpi_M^{-1}(p)$.
	Now given any $\varepsilon>0$, there exists $\delta(\varepsilon)\in(0,\delta)$ such that
	$\abs{\tau(v)}<\varepsilon$ and $e(v)\ge m$ in $M\setminus\overline{\mathcal{B}_{\delta(\varepsilon)}}$.
	Then \eqref{eq:Laplacian_estimate_for_negative_curvature} is applicable since $K_{\tilde{h}}<-1/4$ along
	$\gamma$, and hence
	\begin{equation}
		\label{eq:fine_estimate_Laplacian_of_distance}
		\Delta_g \tilde{d}_0 \ge mf(\tilde{d}_0) - \varepsilon
		\qquad\text{in $M\setminus\overline{\mathcal{B}_{\delta(\varepsilon)}}$},
	\end{equation}
	where
	\begin{equation*}
		f(x)=\frac{\cosh(x/2)-1}{2\sinh(x/2)}.
	\end{equation*}
	Let $d_\varepsilon>0$ is the solution of $f(d_\varepsilon)=2\varepsilon/m$ and take a smooth function
	$\chi_\varepsilon\colon[0,\infty)\longrightarrow[0,\infty)$ such that
	$x\le\chi_\varepsilon(x)\le d_\varepsilon$ for $0\le x\le d_\varepsilon$,
	$\chi_\varepsilon(x)=x$ for $x\ge d_\varepsilon$,
	and the derivative $\chi_\varepsilon^{(k)}(0)$ vanishes for all $k\ge 1$.
	Then $\chi_\varepsilon\circ\tilde{d}_0$ is a smooth function in $M$.
	If $\varphi$ is the function in Lemma \ref{lem:barrier},
	we can take $C(\varepsilon)>0$ so that
	$\Delta_g(\chi_\varepsilon\circ\tilde{d}_0-C(\varepsilon)\varphi)>\varepsilon$ in
	$\overline{\mathcal{B}_{\delta(\varepsilon)}}$.
	Then, from the Omori--Yau generalized maximum principle, there exists $p_0\in M$ for which
	\begin{equation*}
		\chi_\varepsilon\circ\tilde{d}_0(p_0)-C(\varepsilon)\varphi(p_0)\ge
		\sup_M(\chi_\varepsilon\circ\tilde{d}_0-C(\varepsilon)\varphi)-\varepsilon\qquad\text{and}\qquad
		\Delta_g(\chi_\varepsilon\circ\tilde{d}_0-C(\varepsilon)\varphi)(p_0)\le\varepsilon.
	\end{equation*}
	From the definition of $C(\varepsilon)$, $p_0$ should be a point outside
	$\overline{\mathcal{B}_{\delta(\varepsilon)}}$.
	Moreover, $\chi_\varepsilon\circ\tilde{d}_0(p_0)\le d_\varepsilon$ holds.
	In fact, if it is not the case, then $\tilde{d}_0$ is smooth at $p_0$ and
	$\Delta_g\tilde{d}_0(p_0)<\Delta_g(\tilde{d}_0-C(\varepsilon)\varphi)(p_0)\le\varepsilon$,
	which contradicts \eqref{eq:fine_estimate_Laplacian_of_distance}. Consequently,
	\begin{equation*}
		d_\varepsilon+\varepsilon
		\ge \chi_\varepsilon\circ\tilde{d}_0(p_0)+\varepsilon
		\ge \chi_\varepsilon\circ\tilde{d}_0(p)-C(\varepsilon)\varphi(p)
		\ge \tilde{d}_0(p)-C(\varepsilon)\varphi(p)
		\qquad\text{for $p\in M$}.
	\end{equation*}
	Since $\varphi$ uniformly tends to $0$ as $r\to 0$, this shows that so does $\tilde{d}_0$.

	The uniqueness of $u$ is similarly proved as follows.
	If $u_1$, $u_2$ are two such harmonic maps and $\tilde{u}_1$, $\tilde{u}_2$ are relatively homotopic
	lifts of $u_1\circ\varpi_M$, $u_2\circ\varpi_M$,
	then by \eqref{eq:Laplacian_of_difference_of_mappings_simplified},
	$\tilde{d}(\tilde{u}_1,\tilde{u}_2)$ is subharmonic in $M$.
	On the other hand, the argument in the previous paragraph shows that $\tilde{d}(\tilde{u}_1,\tilde{u}_2)\to 0$
	uniformly at $\bdry M$.
	Then by the maximum principle, $\tilde{d}(\tilde{u}_1,\tilde{u}_2)$ must vanish identically.
	
	Finally we prove that $u$ is in $C^1(\overline{M},\overline{N})$. Since this is a local statement,
	it suffices to work on a normal boundary coordinate neighborhood $(\mathcal{U};x^1,\dots,x^m,r)$
	such that $u(\mathcal{U})$ and $v(\mathcal{U})$ are both contained in
	a normal boundary coordinate neighborhood $(\mathcal{V};y^1,\dots,y^n,\rho)$ of $\overline{N}$.
	As $u$ has bounded energy density with respect to $g$ and $h$,
	its energy density with respect to $\overline{g}$ and $\overline{h}$ is also bounded.
	Then by \eqref{eq:tension_divided} and the fact that $u^\infty$ and $v^\infty$ are quasi-equivalent to $r$,
	$\Delta_{\overline{g}}u^j$ and $\Delta_{\overline{g}}v^j$ are both $O(r^{-1})$.
	Let $p\in\mathring{\mathcal{U}}$ be such that the ball $B(p,r_0)$ with respect to the local coordinates
	$(x^1,\dots,x^m,r)$ is relatively compactly contained in $\mathring{\mathcal{U}}$, where $r_0=r(p)/2$.
	Then by the Schauder interior estimate~\cite[Theorem 6.2]{GilbargTrudinger},
	\begin{equation*}
		\norm{u^j}_{1,\alpha;B(p,r_0)}\le Cr_0^{-\alpha}
	\end{equation*}
	for some constant $C>0$ that is independent of $p$.
	Moreover, the uniform convergence $d_0\to 0$ implies $\abs{u^j-v^j}=o(r)$.
	Therefore by the interpolation inequality~\cite[Lemma 6.32]{GilbargTrudinger}, for any $\varepsilon>0$,
	\begin{equation*}
		\sup_{B(p,r_0)}\abs{\nabla_{\overline{g}}(u^j-v^j)}
		\le C\left(\varepsilon + \varepsilon^{-2}r_0^{-1}\sup_{B(p,r_0)}\abs{u^j-v^j}\right),
	\end{equation*}
	where $C>0$ is independent of $p$.
	Therefore, $\nabla_{\overline{g}}(u^j-v^j)=o(1)$ and so $u$ is $C^1$ up to $\mathcal{U}\cap\bdry M$.
\end{proof}

\end{document}